\documentclass[11pt]{article}
\usepackage{amsmath,amssymb,textcomp,stmaryrd,xifthen,psfrag,graphicx,color}
\usepackage{amsfonts,amsthm}
\usepackage{color}

\SetSymbolFont{stmry}{bold}{U}{stmry}{m}{n} 
\usepackage[T1]{fontenc}
\date{}

\newtheorem{theorem}{Theorem}
\newtheorem{lemma}[theorem]{Lemma}
\newtheorem{cor}[theorem]{Corollary}

\theoremstyle{definition} 
\newtheorem{remark}[theorem]{Remark}

\newcommand{\dual}[2]{\langle#1\hspace*{.5mm},#2\rangle}
\newcommand{\vdual}[2]{(#1\hspace*{.5mm},#2)}
\newcommand{\abs}[1]{\vert #1 \vert}
\newcommand{\norm}[3][]{#1\|#2#1\|_{#3}}
\newcommand{\snorm}[2]{|#1|_{#2}}

\newcommand{\wilde}{\widetilde}

\newcommand{\bfone}{\ensuremath{\mathbf{1}}}

\newcommand{\cL}{\ensuremath{\mathcal{L}}}

\newcommand{\R}{\ensuremath{\mathbb{R}}}
\newcommand{\N}{\ensuremath{\mathbb{N}}}

\newcommand{\Ff}{\ensuremath{\mathcal{F}}}

\newcommand{\Pp}{\ensuremath{\mathcal{P}}}

\title{Variational formulation of time-fractional parabolic equations
\thanks{Supported by Conicyt Chile through project FONDECYT 1170672.}}
\author{Michael Karkulik\thanks{Departamento de Matem\'atica, Universidad T\'ecnica Federico Santa Mar\'ia,
  Avenida Espa\~na 1680, Valpara\'iso, Chile,
  \texttt{mkarkulik.mat.utfsm.cl}, email: \texttt{michael.karkulik@usm.cl}}}
\begin{document}
\maketitle
\begin{abstract}
  We consider initial/boundary value problems for time-fractional parabolic PDE
  of order $0<\alpha<1$ with Caputo fractional derivative (also called fractional diffusion equations in the
  literature).
  We prove well-posedness of corresponding variational formulations based entirely
  on fractional Sobolev-Bochner spaces, and clarify the question of possible choices of the initial value.

\end{abstract}
\section{Introduction}
Physical phenomena based on standard diffusion, where the mean square displacement
of a diffusing particle scales linearly with time $\langle x(t)^2\rangle\sim t$,
are typically modeled by partial differential equations involving
standard (i.e., integer order) differential operators. So-called \textit{anomalous diffusion}, on the other
hand, is characterized by non-linear scaling. For example,
a diversive number of systems exhibit anomalous
diffusion which follows the power-law $\langle x(t)^2\rangle \sim t^\alpha$ with $0<\alpha<1$
(subdiffusion) or $1<\alpha <2$ (superdiffusion). Systems with such power-laws include
ones with constrained pathways such as fractal, disordered, or porous media, polymers, aquifers,
and quantum systems, among others.
We refer to~\cite{MetzlerK_RandomWalkGuide_00} for an extensive overview on the subject.
In the latter work, the authors list various ways how to model anomalous diffusion processes.
For problems involving external fields or boundary conditions, the most natural way
is to consider partial differential equations involving so-called \textit{fractional differential operators}.
In the work at hand, we consider a time-fractional parabolic initial/boundary value problem of the form
\begin{align}
  \begin{split}
    \partial_t^\alpha u - \Delta u &= f \quad\text{ in } (0,T)\times\Omega,\\
    u &= 0\quad\text{ on } (0,T)\times \partial\Omega,\\
    u &= g\quad\text{ for } \left\{ 0 \right\}\times\Omega,
  \end{split}
  \label{eq:fde}
\end{align}
where $(0,T)$ is a time interval and $\Omega\subset\R^n$ a spatial Lipschitz domain.
Here, $-\Delta$ is the spatial Laplacian, $1/2<\alpha<1$, and $\partial_t^\alpha$ is
a fractional time derivative of order $\alpha$. More specifically, we will use the so-called
Caputo derivative, which is defined formally by
\begin{align*}
  \partial_t^\alpha u (t) := \frac{1}{\Gamma(1-\alpha)}\int_0^t (t-s)^{-\alpha} u'(s)\,ds.
\end{align*}
Recently, researchers have started to
analyze finite element methods with respect to their ability to approximate
solutions of fractional differential equations.
While this started with classical Galerkin finite element methods for
steady-state fractional diffusion equations as in~\cite{ErvinR_06,JLPR_15},
numerical methods for time-dependent fractional partial differential equations include
time-stepping methods~\cite{ErvinHR_07,JinLZ_13,JinLZ_16},
Discontinuous Galerkin methods~\cite{Mustapha_15,MustaphaNC_16},
as well as methods based on the Laplace transform~\cite{McLeanST_06}.
It goes without saying that this list is far from being exhaustive.
We mention here also the numerical approach
from~\cite{NochettoO_16} which is based on the extension theory by
Caffarelli and Silvestre~\cite{CaffarelliS_07}.
The aforementioned numerical methods are usually based on
a variational formulation of the equation under consideration.
Existing works on variational formulations of
time-fractional parabolic partial differential equations are scarce; as to our knowledge, the
works~\cite{Zacher_09,SakamotoY_11,AllenCV_16} are of relevance in connection with
our model problem~\eqref{eq:fde} (for Semigroup theory for related Volterra integral equations
see~\cite{Pruess_12}).
These works have in common that (i) their functional analytic setting is not based
exclusively on classical Sobolev regularity in time, but rather involves
the operator $\partial_t^\alpha$, and that (ii) the initial value $g$ is taken from $L_2(\Omega)$.
The goal of the present work is to derive the well-posedness of variational formulations set up in classical
Sobolev-Bochner spaces and to clarify the question of regularity needed for the initial data.
Now, as our functional analytic setting is based only on Sobolev regularity, a result of this kind is
specifically interesting for numerical analysis of the equation~\eqref{eq:fde}.
Indeed, approximation results for functions with certain Sobolev regularity are well known and ubiquitous
in numerical analysis.
The property (i) is owed to the fact that there is no rigorous definition of time-fractional
derivatives on fractional Sobolev-Bochner spaces available.
It sure is true that operators defined between real valued Sobolev spaces
$L_2(J)\rightarrow L_2(J)$ do extend to vector-valued counterparts $L_2(J;X)\rightarrow L_2(J;X)$
(for $X$ a Hilbert space, this is a classical result of Marcienkiwicz and Zygmund~\cite{Marcinkiewicz_39}), but
the fact that we are dealing with Sobolev regularity $H^\alpha(J;X)$ in time needs some care and
additional analysis.
To that end, we will show first that the fractional Caputo derivative is a linear and bounded operator
on a time-fractional Sobolev-Bochner space. This way, we can consider a variational formulation of~\eqref{eq:fde}
based exclusively on Sobolev regularity, which resembles classical variational formulations for
parabolic equations.
Regarding the point (ii), the choice of $g\in L_2(\Omega)$ as initial value is indeed admissible,
but one has to bear in mind the following:
While the space $L_2(J;\wilde H^1(\Omega))\cap L_2(J;H^{-1}(\Omega))$, used in variational formulations of
parabolic equations, is continuously embedded in $C(\overline J;L_2(\Omega))$, this is no longer true
for the equation~\eqref{eq:fde}. We will show that the space of solutions of our variational
formulation of~\eqref{eq:fde} is continuously embedded in
$C(\overline J;H^{1-1/\alpha-\varepsilon}(\Omega))$ for all $\varepsilon>0$.
Our main result is then well-posedness of the variational formulation,
cf. Theorem~\ref{thm:var}.


\section{Mathematical setting and main results}
\subsection{Sobolev and Bochner spaces}
We denote by $\Omega\subset\R^d$ a (spatial) Lipschitz domain,
and by $J=(0,T)$ for $T>0$ a temporal interval. We use Lebesgue and Sobolev spaces
$L_2(\Omega)$ and $\wilde H^1(\Omega)$, the tilde denoting vanishing trace on the boundary
$\partial\Omega$. The fractional Sobolev spaces $\wilde H^s(\Omega)$ for $s\in(0,1)$ are
defined by the K-method of interpolation as
$\wilde H^s(\Omega) := [ L_2(\Omega), \wilde H^1(\Omega) ]_{s,2}$, cf.~\cite{Triebel},
where
\begin{align*}
  [B_0,B_1]_{s,2} := \left\{ u\in B_0 \mid \norm{u}{\left[ B_0,B_1 \right]_{s,2}}<\infty \right\},
  \quad \norm{u}{\left[ B_0,B_1 \right]_{s,2}}^2 := \int_0^\infty t^{-2s-1}K_{[B_0,B_1]}(t,u)^2\,dt
\end{align*}
with the K-functional
\begin{align*}
  K_{[B_0,B_1]}(t,u)^2 := \inf_{w\in B_1}\norm{u-w}{B_0}^2 + t^2\norm{w}{B_1}^2.
\end{align*}
The topological dual of a Banach space $X$ is denoted by $X'$, and we define
$H^{-s}(\Omega) := \wilde H^s(\Omega)'$
as the topological duals with respect to the extended $L_2(\Omega)$ scalar product $\vdual{\cdot}{\cdot}$,
and duality will be denoted by $\dual{\cdot}{\cdot}$.
We set $\widetilde H^0(\Omega):=L_2(\Omega)$.
In time, we additionally use Lebesgue and Sobolev spaces $L_2(J)$ and $H^\alpha(J)$, for $\alpha\in(0,1]$.
The $L_2(J)$ scalar product will also be denoted by $\vdual{\cdot}{\cdot}$, but it will always be clear
which scalar product we are using.
We use the notation $Du$ to denote the distributional derivative of a function $u$ given on $J$.
For $\alpha\in(0,1)$ the norm on the space $H^\alpha(J)$ is given by
\begin{align*}
  \norm{f}{H^\alpha(J)}^2 := \norm{f}{L_2(J)}^2 + \snorm{f}{H^\alpha(J)}^2 < \infty,
  \qquad\text{ where }
  \snorm{f}{H^\alpha(J)}^2 := \int_J\int_J \frac{\abs{f(s)-f(t)}^2}{\abs{s-t}^{2\alpha+1}}\,ds\,dt.
\end{align*}
We mention that $H^\alpha(J) = [L_2(J),H^1(J)]_{\alpha,2}$ with equivalent norms.
We set $H^{-s}(J):=\wilde H^s(J)'$ and $\wilde H^{-s}(J):=H^s(J)$ for $s\in(0,1)$.
For a Banach space $X$, we use
the Bochner space $L_2(J;X)$ of functions $f:J\rightarrow X$ which are strongly measurable with
respect to the Lebesgue measure $ds$ on $J$ and
\begin{align*}
  \norm{f}{L_2(J;X)}^2 := \int_{J} \norm{f(s)}{X}^2\,ds < \infty.
\end{align*}
For $w$ a measurable, positive function on $J$, we denote by $L_2(J,w;X)$ the $w$-weighted Lebesgue space
of strongly measurable functions with norm
\begin{align*}
  \norm{f}{L_2(J,w;X)}^2 := \int_{J} w(s)^2\norm{f(s)}{X}^2\,ds < \infty.
\end{align*}
We say that a function $f\in L_2(J;X)$ has a weak time-derivative $\partial_tf\in L_2(J;X)$, if
\begin{align}\label{def:dt}
  \int_J \partial_tf \varphi = - \int_J f \partial_t\varphi \quad\text{ for all } \varphi\in C_0^\infty(J).
\end{align}
Note that this last integral has to be understand in the sense of Bochner. We define the space
$H^1(J;X)$ as the space of functions with
\begin{align*}
  \norm{f}{H^1(J;X)}^2 := \norm{f}{L_2(J;X)}^2 + \norm{\partial_tf}{L_2(J;X)}^2.
\end{align*}
For $0 < \alpha < 1$, we also use the fractional
Sobolev-Bochner space $H^\alpha(J;X)$ of $ds$-strongly measurable
functions $f:J\rightarrow X$ with
\begin{align*}
  \norm{f}{H^\alpha(J;X)}^2 := \norm{f}{L_2(J;X)}^2 + \snorm{f}{H^\alpha(J;X)}^2 < \infty,
  \qquad\text{ where }
  \snorm{f}{H^\alpha(J;X)}^2 := \int_J\int_J \frac{\norm{f(s)-f(t)}{X}^2}{\abs{s-t}^{2\alpha+1}}\,ds\,dt.
\end{align*}
We will also use these Bochner spaces on $\R$ instead of $J$. For a recent introduction to
Bochner spaces, we refer to~\cite{HNVW_16}.
\subsection{Fractional time derivative on Bochner spaces}
For $0<\beta<1$, we define the left and right-sided Riemann-Liouville fractional integral
operators
\begin{align*}
  _0D^{-\beta} u(t) := \frac{1}{\Gamma(\beta)} \int_0^t(t-s)^{\beta-1}u(s)\,ds
  \quad\text{ and }\quad
  D^{-\beta}_T u(t) := \frac{1}{\Gamma(\beta)} \int_t^{T} (s-t)^{\beta-1}u(s)\,ds,
\end{align*}
where $\Gamma$ is Euler's Gamma function.
For sufficiently smooth functions $u$, the left-sided Caputo
fractional derivative $\partial^\alpha$ for $\alpha\in(0,1)$ is defined as
$\partial^\alpha u := {_0D}^{\alpha-1}Du$. We will show below in Lemma~\ref{lem:ext2bochner} that
the tensorised version $_0D^{-\beta}\otimes I$ defined by
$(_0D^{-\beta}\otimes I)(u\otimes x) := (_0D^{-\beta}u)\otimes x$
can be extended uniquely to a linear and bounded operator
$_0D^{-\beta}:L_2(J;X)\rightarrow H^\beta(J;X)$ for a Hilbert space $X$.
This allows us to prove the following result. The proof will be carried
out below in Section~\ref{proofs}.
\begin{theorem}\label{thm:fracder}
  Let $\partial_t$ be the weak time derivative defined in~\eqref{def:dt}.
  Then, for $\alpha\in(1/2,1)$, the operator $\partial_t^\alpha := {_0D}^{\alpha-1}\circ \partial_t$
  is linear and bounded as $\partial_t^\alpha: H^\alpha(J;H^{-1}(\Omega))\rightarrow L_2(J;H^{-1}(\Omega))$.
  \qed
\end{theorem}
\subsection{Variational formulation and main result}
Our variational formulation of~\eqref{eq:fde} is the following: Given $f\in L_2(J;H^{-1}(\Omega))$ and
$g\in H^{1-1/\alpha + \delta}(\Omega)$ for some $\delta>0$,
find $u\in L_2(J;\wilde H^1(\Omega))$ with $u\in H^\alpha(J;H^{-1}(\Omega))$ such that
\begin{align}\label{eq:var}
  \dual{\partial^\alpha_t u}{v} + \vdual{\nabla u}{\nabla v} = \vdual{f}{v} \quad\text{ for all }
  v\in \wilde H^1(\Omega),
\end{align}
almost everywhere in $J$, and $u(0,\cdot) = g(\cdot)$. The duality $\dual{\partial_t^\alpha u}{v}$
in~\eqref{eq:var} makes sense
due to the mapping properties of $\partial_t^\alpha$ from Theorem~\ref{thm:fracder},
and the initial condition makes makes sense as we will show in Corollary~\ref{cor:embedding} below that
$L_2(J;\wilde H^1(\Omega))\cap H^\alpha(J;H^{-1}(\Omega))$ is continuously embedded in
$C(\overline J;H^{1-1/\alpha - \varepsilon}(\Omega))$ for all $\varepsilon>0$.
The following theorem is our main result and will be proven below in Section~\ref{proofs}.
\begin{theorem}\label{thm:var}
  The variational formulation~\eqref{eq:var} is well posed: there exists a unique solution $u$, and
  \begin{align*}
    \norm{u}{L_2(J;\wilde H^1(\Omega))} + \norm{u}{H^\alpha(J;H^{-1}(\Omega))}
    \leq C_\delta \left( 
    \norm{g}{H^{1-1/\alpha+\delta}(\Omega)} + \norm{f}{L_2(J;H^{-1}(\Omega))}
    \right).
  \end{align*}
  The constant $C_\delta>0$ depends only on $\delta$.
  \qed
\end{theorem}
\begin{remark}
  It is textbook knowledge that there holds the continuous embedding
  \begin{align*}
    L_2(J;\wilde H^1(\Omega))\cap H^1(J;H^{-1}(\Omega))\hookrightarrow C(\overline J;L_2(\Omega)).
  \end{align*}
  In the present case, we have the embedding
  \begin{align*}
    L_2(J;\wilde H^1(\Omega))\cap H^{\alpha}(J;H^{-1}(\Omega))\hookrightarrow
    C(\overline J;H^{1-1/\alpha-\varepsilon}(\Omega))
  \end{align*}
  for all $\varepsilon>0$, cf. Lemma~\ref{lem:embedding} below. The reason for the missing power of
  $\varepsilon$ is that we use the embedding result
  $H^{1/2+\varepsilon}(J;X)\hookrightarrow C(\overline J;X)$.
  Furthermore, note that the stability
  estimate of Theorem~\ref{thm:var} involves the $H^{1-1/\alpha +\delta}(\Omega)$ norm of
  the initial data $g$, and the constant $C_\delta$ is expected to blow up for $\delta\rightarrow 0$.
\end{remark}
\section{Technical results}
\subsection{Fractional integral and differential operators}
We have the following results. The first point is an extension
of a recent result in~\cite{JLPR_15} and can be found in~\cite[Lemma~5]{EFHK_17},
while the second point is part of the proof of~\cite[Lemma~6]{EFHK_17}.
The third part can be found in~\cite[Lem.~2.7]{KilbasST_06}.
\begin{lemma}\label{lem:rl}
  \begin{itemize}
    \item[(i)] For every $s\in\R$ with $-\beta\leq s$ and $\beta>0$, the operators $_0D^{-\beta}$
      and $D_T^{-\beta}$ can be extended to bounded linear operators
      $\wilde H^{s}(J)\rightarrow H^{s+\beta}(J)$.
    \item[(ii)] For $0<\beta<1$, the operator $_0D^{-\beta}$ is elliptic on $H^{-\beta/2}(J)$.
    \item[(iii)] For $0<\beta<1$ and $u,v\in L_2(J)$ it holds
      $\vdual{D_0^{-\beta}u}{v} = \vdual{u}{D_T^{-\beta}v}$.
  \end{itemize}
  \qed
\end{lemma}
The Mittag-Leffler function arises naturally in the study of fractional differential equations.
We refer to~\cite[Section~18.1]{ErdelyiMOT_81} for an overview.
It is defined as
\begin{align*}
  E_{n_1,n_2}(z) := \sum_{k=0}^\infty \frac{z^k}{\Gamma(kn_1+n_2)}.
\end{align*}
According to~\cite[Thm.~1.6]{Podlubny}, for $z\in\R$,
\begin{align}\label{eq:mf}
  E_{n_1,n_2}(z) \lesssim \frac{1}{1+\abs{z}},
\end{align}
and due to~\cite[Thm.~4.3]{Diethelm_FracDiff10},
\begin{align}\label{eq:d:eig}
  _0D^{\alpha-1}D E_{\alpha,1}(\lambda t^\alpha) = \lambda E_{\alpha,1}(\lambda t^\alpha).
\end{align}
Furthermore, by~\cite{Schneider_96}, $E_{\alpha,1}(-z)$ is completely monotone for $0<\alpha\leq 1$ and positive $z$,
in particular,
\begin{align}\label{eq:mon}
  E_{\alpha,1}'(-z) \geq 0 \quad\text{ for positive } z. 
\end{align}
We will need the following result on fractional seminorms, which combines the $H^s$ norm
and the dual norm of the distributional derivative.
\begin{lemma}\label{lem:poincare}
  Let $s\in (0,1)$ be fixed. There holds
  \begin{align*}
    \snorm{u}{H^s(J)} \lesssim \norm{Du}{H^{s-1}(J)}
    \quad\text{ for all } u\in H^s(J),
  \end{align*}
  where $Du$ is the distributional derivative of $u$.
\end{lemma}
\begin{proof}
  As $u\in L_2(J)$, it holds $Du\in H^{-1}(J)$.
  We can write $u=D\psi + c$ with $c\in\R$, where
  $\psi\in \wilde H^1(J)$ is the unique solution of
  $\vdual{D\psi}{D\varphi}=\vdual{u}{D\varphi}$ for all $\varphi\in \wilde H^1(J)$.
  Then, $\norm{\psi}{\wilde H^1(J)}\lesssim\norm{u}{L_2(J)}$,
  and due to the definition of the distributional derivative we see
  \begin{align*}
    \abs{\vdual{u}{D\psi}} = 
    \abs{\vdual{Du}{\psi}} \lesssim \norm{Du}{H^{-1}(J)}\norm{\psi}{\wilde H^1(J)}
    \lesssim \norm{Du}{H^{-1}(J)} \norm{u}{L_2(J)}.
  \end{align*}
  We conclude that for $u\in L_2(J)$, it holds
  \begin{align*}
    \norm{u}{L_2(J)}^2 = \vdual{u}{D\psi} + \vdual{u}{c}
    \lesssim \norm{Du}{H^{-1}(J)} \norm{u}{L_2(J)} + \vdual{u}{c}.
  \end{align*}
  Now we apply this estimate to $u-\overline u$, where $\overline u$ denotes the mean value
  of $u$, and obtain
  \begin{align}\label{lemma:poincare:-1}
    \norm{u-\overline u}{L_2(J)} \lesssim \norm{Du}{H^{-1}(J)}.
  \end{align}
  The standard Poincar\'e inequality states that
  \begin{align}\label{lemma:poincare:0}
    \norm{u-\overline u}{H^1(J)} \lesssim \norm{Du}{L_2(J)}.
  \end{align}
  The $H^s(J)$ norm can equivalently be obtained by the K-method of
  interpolation via
  \begin{align*}
    \norm{u-\overline u}{H^s(J)}^2 \simeq \norm{u-\overline
    u}{[L_2(J),H^1(J)]_{s,2}}^2 =
    \int_0^\infty t^{-2s} \left(
    \inf_{v\in H^1(J)} \norm{u-\overline{u}-v}{L_2(J)}^2 + t^2 \norm{v}{H^1(J)}^2
    \right) \frac{dt}{t}.
  \end{align*}
  Using~\eqref{lemma:poincare:-1} and~\eqref{lemma:poincare:0}, we obtain
  \begin{align*}
    \inf_{v\in H^1(J)} \norm{u-\overline{u}-v}{L_2(J)}^2 + t^2 \norm{v}{H^1(J)}^2
    &\leq
    \inf_{\substack{v\in H^1(J)\\\overline v=0}} \norm{u-\overline{u}-v}{L_2(J)}^2 + t^2
    \norm{v}{H^1(J)}^2\\
    &\lesssim
    \inf_{\substack{v\in H^1(J)\\\overline v=0}}
    \norm{Du-Dv}{H^{-1}(J)}^2 + t^2 \norm{Dv}{L_2(J)}^2
  \end{align*}
  Next we use that for $w\in L_2(J)$ there is a $\psi\in H^1(J)$ with $\overline\psi=0$ such that
  $D\psi=w$. We conclude
  \begin{align*}
    \norm{u-\overline u}{H^s(J)}^2 \lesssim
    \int_0^\infty t^{-2s} \left(
    \inf_{w\in L_2(J)} \norm{Du-w}{H^{-1}(J)}^2 + t^2
    \norm{w}{L_2(J)}^2
    \right) \frac{dt}{t}.
  \end{align*}
  By definition, the right-hand side is
  $\norm{Du}{[H^{-1}(J),L_2(J)]_{s,2}}^2$, which
  is equivalent to $\norm{Du}{H^{s-1}(J)}^2$.
  This concludes the proof.
\end{proof}
The next lemma establishes a norm equivalence on a fractional Sobolev space.
\begin{lemma}\label{lem:equivalence}
  Let $1/2<s<1$. Then, for all $u\in H^s(J)$,
  \begin{align*}
    \snorm{u}{H^s(J)}\sim \norm{_0D^{s-1}Du}{L_2(J)}.
  \end{align*}
\end{lemma}
\begin{proof}
  We have
  \begin{align*}
    \norm{_0D^{s-1}Du}{L_2(J)} \lesssim \norm{Du}{\wilde H^{s-1}(J)}
    \lesssim \norm{Du}{H^{s-1}(J)}\lesssim \norm{u}{H^s(J)}.
  \end{align*}
  Here, the first estimate follows from Lemma~\ref{lem:rl}, and the second one can be
  found in~\cite[Lem.~5]{Heuer}. To see the third estimate, recall that $D$ is the
  distributional derivative, and hence $\norm{Du}{H^{-1}(J)} \leq \norm{u}{L_2(J)}$
  as well as $\norm{Du}{L_2(J)} \leq \norm{u}{H^1(J)}$. The third estimate now follows
  from an interpolation argument.
  The fact that $Du = D(u-\overline u)$ for the mean value $\overline u$ of $u$
  and Poincare's inequality show
  \begin{align*}
    \snorm{u}{H^s(J)}\gtrsim \norm{_0D^{s-1}Du}{L_2(J)}.
  \end{align*}
  To show the converse estimate, we take $u\in C^\infty(\overline J)$ and estimate
  with Lemmas~\ref{lem:poincare} and~\ref{lem:rl}
  \begin{align*}
    \snorm{u}{H^s(J)}^2 &\lesssim \norm{Du}{H^{s-1}(J)}^2 \lesssim \vdual{_0D^{2(s-1)}Du}{Du}\\
    &= \vdual{_0D^{s-1}Du}{D_T^{s-1}Du} \leq
    \norm{_0D^{s-1}Du}{L_2(J)}\norm{D_T^{s-1}Du}{L_2(J)}.
  \end{align*}
  Here, the identity follows from Lemma~\ref{lem:rl}, $(iii)$.
  Due to Lemma~\ref{lem:rl}, it also holds
  $\norm{D_T^{s-1}Du}{L_2(J)}\lesssim \norm{Du}{\widetilde H^{s-1}(J)}\lesssim \norm{u}{H^s(J)}$,
  where the second estimate wa already shown at the beginning of this proof.
  Applying the whole argument to $u-\overline u$ and using Poincare's inequality finally shows the statement.
\end{proof}
\subsection{Sobolev and Bochner spaces}
For $s\in(-1,0]$ we have the interpolation estimate
\begin{align}\label{eq:int}
  \norm{u}{H^s(\Omega)} \leq C(s) \norm{u}{H^{-1}(\Omega)}^{(1-s)/2} \cdot
  \norm{u}{\wilde H^1(\Omega)}^{(1+s)/2},
\end{align}
with $C(s)>0$ a constant depending only on $s$.
This estimate follows for $s=0$ by duality,
and for $s\in (-1,0)$ using additionally~\cite[1.3.3 (g)]{Triebel} and the fact
that duality and interpolation commute, cf.~\cite[1.11.2]{Triebel}.
For a measurable set $M\subset \R$, we denote by $\bfone_M$ the characteristic function on $M$, and
for a function $\phi:J\rightarrow \R$ and $x\in X$, we define $\phi\otimes x:=J\rightarrow X$
as $\left( \phi\otimes x \right)(s):= \phi(s)x$. We denote by $\cL(J)$ the $\sigma$-algebra of all
$ds$-measurable sets on $J$, and by $S(J)$ the set of simple functions.
It is known that the following subsets are dense for $J$ bounded or $J=\R$,
\begin{align*}
  S(J;X) &:= \left\{ \sum_{i=1}^n \bfone_{A_i} \otimes x_i \mid A_i \in \cL(J), x_i\in X
  \text{ for all } i=1,\dots, n, n\in\N \right\} \subset L_2(J;X),\\
  S^\infty(J;X) &:= \left\{ \sum_{i=1}^n \varphi_i \otimes x_i \mid \varphi_i\in C_c^\infty(\R),
  x_i\in X \text{ for all } i=1,\dots, n, n\in\N \right\} \subset H^1(J;X).
\end{align*}
We assume from now on that the Banach spaces $X$ are
reflexive; this implies that they have the so-called Radon-Nikod\'ym property,
cf.~\cite[Thm.~1.95]{HNVW_16},
which is sufficient and necessary
in order to have that $L_2(J;X)'$ is isometrically isomorphic to $L_2(J;X')$,
cf.~\cite[Thm.~1.84]{HNVW_16}.
We can extend $\partial_t$ for $u \in L_2(J;X)$ by defining $\partial_t u\in H^1_0(J;X')'$ as
\begin{align*}
  -\int_J \dual{u(s)}{\partial_t\varphi(s)}\,ds\quad\text{ for }\varphi\in H^1_0(J;X').
\end{align*}
Then, we have that $\partial_t:L_2(J;X)\rightarrow H^1_0(J;X')'$ is bounded.
Furthermore, $\partial_t:H^1(J;X)\rightarrow L_2(J;X) = L_2(J;X')'$ is bounded, and
by interpolation, we have that for $s\in(0,1)$
\begin{align}\label{eq:dt}
  \partial_t: \left[ L_2(J;X), H^1(J;X) \right]_s \rightarrow \left[ H^1_0(J;X')', L_2(J;X')' \right]_s
\end{align}
is bounded.
We will need the following results on interpolation of Sobolev-Bochner spaces.
\begin{lemma}\label{lem:interpolation}
  There holds
  \begin{align*}
    \left[ L_2(\R;X), H^1(\R;X) \right]_s = H^s(\R;X)
  \end{align*}
  and
  \begin{align*}
    \left[ \wilde H^1(J;X')', L_2(J;X')' \right]_s = \left[ \wilde H^1(J;X'), L_2(J;X') \right]_s'
    = \left[ L_2(J;X'), \wilde H^1(J;X') \right]_{1-s}'
    = \wilde H^{1-s}(J;X')'.
  \end{align*}
\end{lemma}
\begin{proof}
The first identity is due to~\cite[Thm.~2.91]{HNVW_16}, and
the second and third identities are well-known results in interpolation theory,
cf.~\cite[1.11.2]{Triebel}.
The last identity is a variant of the first one with bounded interval and
zero traces. Using extension theorems, its proof can in fact be reduced to the first identity.
In the case of scalar-valued Sobolev spaces, we
refer to~\cite[Thm.~14.2.3]{BrennerScott_08} for details.
\end{proof}
Next, we will establish continuous embeddings for the function space of our variational formulation.
\begin{lemma}\label{lem:embedding}
  Suppose that $\alpha\in (0,1)$, $s\in(-1,0]$ and $0<r$ are such that $r<\alpha(1-s)/2$.
  Then, we have the continuous embedding
  \begin{align*}
    L_2(J;\widetilde H^1(\Omega))\cap H^\alpha(J;H^{-1}(\Omega))
    \hookrightarrow
    H^r(J;H^s(\Omega)).
  \end{align*}
\end{lemma}
\begin{proof}
  It is clear that $\norm{u}{L_2(J;H^s(\Omega))}\leq \norm{u}{L_2(J;\wilde H^1(\Omega))}$.
  To bound the $H^r$-seminorm, we write for $r<\alpha(1-s)/2$
  \begin{align*}
    2r+1 = (2\alpha+1)\frac{1-s}{2} + (1-\varepsilon)\frac{1+s}{2}
  \end{align*}
  for some $\varepsilon>0$. The interpolation estimate~\eqref{eq:int} and
  the inequalities of Cauchy-Schwarz and Young then yield
  \begin{align*}
    \int_J\int_J \frac{\norm{u(s)-u(t)}{\wilde H^s(\Omega)}^2}{\abs{s-t}^{2r+1}}\,dtds
    &\lesssim \int_J\int_J \frac{\norm{u(s)-u(t)}{H^{-1}(\Omega)}^{1-s}\norm{u(s)-u(t)}{\wilde H^1(\Omega)}^{1+s}}{\abs{s-t}^{2r+1}}\,dtds\\
    &\lesssim \int_J \left( \int_J \frac{\norm{u(s)-u(t)}{H^{-1}(\Omega)}^2}{\abs{s-t}^{2\alpha+1}}\,dt\right)^{(1-s)/2} \\
    &\qquad\qquad\left( \int_J \frac{\norm{u(s)-u(t)}{\wilde H^1(\Omega)}^2}{\abs{s-t}^{1-\varepsilon}}\,dt \right)^{(1+s)/2}\,ds\\
    &\lesssim \int_J \int_J \frac{\norm{u(s)-u(t)}{H^{-1}(\Omega)}^2}{\abs{s-t}^{2\alpha+1}}\,dtds
    + \int_J \int_J \frac{\norm{u(s)-u(t)}{\wilde H^1(\Omega)}^2}{\abs{s-t}^{1-\varepsilon}}\,dtds\\
  \end{align*}
  and as $\varepsilon>0$, the last integral can be bounded by $\norm{u}{L_2(J;\wilde H^1(\Omega))}$.
\end{proof}
\begin{cor}\label{cor:embedding}
Suppose that $\alpha\in (0,1)$ and $s\in(-1,0]$ are such that $s<1-1/\alpha$.
Then, we have the continuous embedding
  \begin{align*}
    L_2(J;\widetilde H^1(\Omega))\cap H^\alpha(J;H^{-1}(\Omega))
    \hookrightarrow
    C(\overline J;H^s(\Omega)).
  \end{align*}
\end{cor}
\begin{proof}
  If $s<1-1/\alpha$, then $1/2<\alpha(1-s)/2$, and according to Lemma~\ref{lem:embedding}
  there holds the continuous embedding
  $L_2(J;\widetilde H^1(\Omega))\cap H^\alpha(J;H^{-1}(\Omega))\hookrightarrow H^{1/2+\varepsilon}(J;H^s(\Omega))$
  for a sufficiently small $\varepsilon>0$.
  According to~\cite[Thm.~2.95]{HNVW_16} there also holds the continuous
  embedding $H^{1/2+\varepsilon}(J;H^s(\Omega))\hookrightarrow C(\overline J;H^{s}(\Omega))$,
  and this proves the statement.
\end{proof}
The next lemma shows that the Riemann-Liouville fractional integral operators can be extended
in the canonical way (i.e., by tensorisation) to Sobolev-Bochner spaces.
\begin{lemma}\label{lem:ext2bochner}
  Suppose that $X$ is a Hilbert space and $0<\beta<1/2$. Then, the operator
  \begin{align*}
    _0D^{-\beta}\otimes I:=
    \begin{cases}
      S(J;X) \rightarrow H^\beta(J;X)\\
      \sum_{i=1}^n x_i \bfone_{A_i} \mapsto \sum_{i=1}^n (_0D^{-\beta}\bfone_{A_i}) x_i
    \end{cases}
  \end{align*}
  can be extended uniquely to a linear and bounded operator
  $_0D^{-\beta}:L_2(J;X)\rightarrow H^\beta(J;X)$.
  The same statement is true for the operator $D_T^{-\beta}\otimes I$.
\end{lemma}
\begin{proof}
  It follow from Lemma~\ref{lem:rl} (i) that the operator $_0D^{-\beta}: L_2(J)\rightarrow L_2(J)$
  is bounded.
  Furthermore, it is a positive operator, i.e., $_0D^{-\beta}(\bfone_{A_i})\geq 0$ on $J$.
  It is then easy to see, cf.~\cite[Thm.~2.3]{HNVW_16}, that
  \begin{align}\label{eq:D:L2bound}
    \norm{_0D^{-\beta}\otimes I u}{L_2(J;X)} \leq \norm{_0D^{-\beta}}{L_2(J)\rightarrow L_2(J)}
    \norm{u}{L_2(J;X)}
    \quad \text{ for } u\in S(J;X),
  \end{align}
  and as $S(J;X)$ is dense in $L_2(J;X)$, we obtain boundedness $_0D^{-\beta}:L_2(J;X)\rightarrow L_2(J;X)$.
  Next, we will follow the ideas developed in~\cite[Thm.~3.1]{JLPR_15}.
  Denoting by $\wilde f\in L_2(\R)$ the extension of $f$ by zero, it holds
  $_0D^{-\beta}f (x) = {_{-\infty}D}^{-\beta}\wilde f (x)$.
  Denote by $\Ff:L_2(\R)\rightarrow L_2(\R)$ the Fourier transformation.
  Then, the operator $\Ff\otimes I$ extends to an isometry $\Ff:L_2(\R;X)\rightarrow L_2(\R;X)$:
  For general operators, this is a classical result by Marcienkiwicz and Zygmund~\cite{Marcinkiewicz_39},
  cf.~\cite[Thm.~2.9]{HNVW_16}, but
  in the present case of the Fourier transformation it can be seen readily by using density of simple functions
  $S(\R;X)$ in $L_2(\R;X)$ and the Plancherel theorem for the scalar-valued Fourier transformation.
  Furthermore, for $u\in L_2(\R;X)$ we have $\Ff\Ff u = \Pp u$ with
  $\Pp u (x) = u(-x)$ the parity operator.
  For a function $\varphi=\sum_{i=1}^n\varphi_i\otimes x_i\in S^\infty(\R;X)$,
  we conclude
  \begin{align*}
    \norm{\varphi}{H^1(\R;X)}^2 &=
    \norm{\sum_{i=1}^n\varphi_i\otimes x_i}{L_2(\R;X)}^2 +
    \norm{\sum_{i=1}^n\partial_t\varphi_i\otimes x_i}{L_2(\R;X)}^2\\
    &= 
    \norm{\sum_{i=1}^n\Ff\varphi_i\otimes x_i}{L_2(\R;X)}^2 +
    \norm{\sum_{i=1}^n\Ff(\partial_t\varphi_i)\otimes x_i}{L_2(\R;X)}^2\\
    &=
    \norm{g\Ff\varphi}{L_2(\R;X)}^2,
  \end{align*}
  with weight function $g(\omega) := \sqrt{1+\omega^2}$. By density, this shows that
  $\Ff\otimes I$ can be extended to an isometry $\Ff:H^1(\R,X)\rightarrow L_2(\R,g;X)$.
  By interpolation and Lemma~\ref{lem:interpolation}, we conclude that
  $\Ff:H^s(\R;X)\rightarrow L_2(\R,g^s;X)$ is bounded. To show that this operator is an isometry,
  consider a decomposition $\wilde u_0 + \wilde u_1 = \Ff u$ with
  $\wilde u_0\in L_2(\R;X), \wilde u_1\in L_2(\R,g;X)$. From
  $\norm{\Ff \wilde u_1}{H^1(\R;X)} = \norm{\Ff \Ff \wilde u_1}{L_2(\R,g;X)} = \norm {\wilde u_1}{L_2(\R,g;X)}$
  we conclude $\Ff \wilde u_1\in H^1(\R;X)$ and due to $\Ff \wilde u_0 + \Ff \wilde u_1 = \Ff\Ff u = \Pp u$
  we have that $\Pp\Ff \wilde u_0 + \Pp \Ff \wilde u_1 = u$ is a decomposition of $u$. Hence,
  \begin{align*}
    \norm{\wilde u_0}{L_2(\R;X)}^2 + t \norm{\wilde u_1}{L_2(\R,g;X)}^2
    &= \norm{\Ff\wilde u_0}{L_2(\R;X)}^2 + t \norm{\Ff\wilde u_1}{H^1(\R;X)}^2\\
    &= \norm{\Pp\Ff\wilde u_0}{L_2(\R;X)}^2 + t \norm{\Pp\Ff\wilde u_1}{H^1(\R;X)}^2,
  \end{align*}
  which implies $K_{[L_2(\R;X),H^1(\R;X)]}(t,u)^2 \leq K_{[L_2(\R;X),L_2(\R,g;X)]}(t,\Ff u)$.
  This shows that $\Ff:H^s(\R;X)\rightarrow L_2(\R,g^s;X)$ is an isometry.
  Next, for a simple function $u\in S(\R;X)$,
  \begin{align*}
    \norm{_{-\infty}D^{-\beta}\otimes Iu}{H^s(\R;X)}^2 &=
    \int_{\R}g(\omega)^{2s} \norm{\Ff _{-\infty}D^{_\beta}u(\omega)}{X}^2\,d\omega\\
    &\lesssim \int_{\abs{\omega}\leq 1} \norm{\Ff _{-\infty}D^{_\beta}u(\omega)}{X}^2\,d\omega
    + \int_{\abs{\omega}> 1} (\omega^{-2}+1)^s\norm{\Ff u(\omega)}{X}^2\,d\omega\\
    &\leq \int_{\R} \norm{\Ff _{-\infty}D^{_\beta}u(\omega)}{X}^2\,d\omega
    + \int_{\abs{\omega}>1} \norm{\Ff u(\omega)}{X}^2\,d\omega\\
    &\lesssim  \int_{\R} \norm{_{-\infty}D^{_\beta}u(s)}{X}^2\,ds
    + \int_{\R} \norm{u(s)}{X}^2\,ds
    \lesssim \norm{u}{L_2(\R;X)}^2,
  \end{align*}
  and by density we get the desired result. The proof for $D_T^{-\beta}$ follows along
  the same lines.
\end{proof}
\begin{lemma}\label{lem:ext2}
  The operator $_0D^{-\beta}u\otimes I$ has a unique extension as bounded and linear operator
  \begin{align*}
    _0D^{-\beta}\otimes I: H^\beta(J;\wilde H^1(\Omega))' \rightarrow L_2(J;H^{-1}(\Omega)).
  \end{align*}
\end{lemma}
\begin{proof}
  For $u=\sum_{i=1}^n \bfone_{A_i}\otimes u_i\in S(J;H^{-1}(\Omega))$,
  $v=\sum_{i=1}^n \bfone_{A_i}\otimes v_i\in S(J;\wilde H^1(\Omega))$, we compute
  \begin{align}\label{eq:lem:ext2}
    \begin{split}
    \vdual{u}{D^{-\beta}_T\otimes I v} &= \int_J \dual{u(s)}{D^{-\beta}_T\otimes I v(s)}\,ds\\
    &= \sum_{i,j=1}^n \dual{u_i}{v_i}\int_J \bfone_{A_i}(s) D^{-\beta}_T\bfone_{A_j}(s)\,ds\\
    &= \sum_{i,j=1}^n \dual{u_i}{v_i}\int_J {_0D^{-\beta}}\bfone_{A_i}(s) \bfone_{A_j}(s)\,ds\\
    &= \vdual{_0D^{-\beta}\otimes Iu}{v}.
    \end{split}
  \end{align}
  As $H^\beta(J;\wilde H^1(\Omega))$ is dense in $L_2(J;\wilde H^1(\Omega))$,
  $L_2(J;H^{-1}(\Omega)) = L_2(J;\wilde H^1(\Omega))'$ is dense in $H^\beta(J;\wilde H^1(\Omega))'$.
  According to Lemma~\ref{lem:ext2bochner},
  $D_T^{-\beta}:L_2(J;\wilde H^1(\Omega))\rightarrow H^\beta(J;\wilde H^1(\Omega))$ is bounded,
  and hence the equality~\eqref{eq:lem:ext2} shows that $_0D^{-\beta}$ can be extended 
  as stipulated. This finishes the proof.
\end{proof}
\subsection{Proof of the main theorems}\label{proofs}
\begin{proof}[Proof of Theorem~\ref{thm:fracder}]
  Using the boundedness~\eqref{eq:dt} of $\partial_t$ and
  Lemmas~\ref{lem:interpolation} and~\ref{lem:ext2}, we conclude that
\begin{align*}
  _0D^{\alpha-1}\circ\partial_t:
  H^\alpha(J;H^{-1}(\Omega)) \rightarrow \wilde H^{1-\alpha}(J;\wilde H^1(\Omega))'
  = H^{1-\alpha}(J;\wilde H^1(\Omega))' \rightarrow L_2(J;H^{-1}(\Omega))
\end{align*}
is bounded.
\end{proof}
\begin{proof}[Proof of Theorem~\ref{thm:var}]
  We mimic the proof for parabolic PDE.
  Take $(w_k)_{k\geq1}$ the $L_2(\Omega)$-orthonormal basis of eigenfunctions
  and $(\lambda_k)_{k\geq1}$ the eigenvalues of $-\Delta$. We make the ansatz
  \begin{align*}
    u_m(t) := \sum_{k=1}^m d_m^k(t)w_k.
  \end{align*}
  Now, we are looking for $d_m^k:J\rightarrow \R$ such that
  \begin{align}\label{thm:eq:fde}
    \begin{split}
    \partial_t^\alpha d_m^k(t) + \lambda_k d_m^k(t)
    &= \dual{f(t)}{w_k},\quad k=1,\dots,m,\\
    d_m^k(0) &= \dual{g}{w_k},\quad k=1,\dots,m.
    \end{split}
  \end{align}
  According to~\cite[Thm.~2.1]{Barrett_54}, cf.~\cite[Thm.~7.2]{Diethelm_FracDiff10}
  and~\cite[Chapter~3.1]{KilbasST_06},
  the solutions to these equations are given uniquely by
  \begin{align*}
    d_m^k(t) = \dual{g}{w_k} E_\alpha(-\lambda_k t^\alpha)
    + \psi_k(t), \quad k=1,\dots,m,
  \end{align*}
  where $\psi_k(t) := \alpha \int_0^t \dual{f(t-s)}{w_k} s^{\alpha-1}E_\alpha'(-\lambda_ks^\alpha)\,ds$.
  In order to obtain energy estimates
  for the $u_m$, we can extend the calculations carried out in~\cite{SakamotoY_11}.
  However, as we aim at weaker initial values, we need a finer analysis.
  First, using the bound~\eqref{eq:mf}, we have for $\varepsilon\in[0,1]$
  \begin{align*}
    \abs{E_{\alpha,1}(z)} \lesssim \frac{1}{1+\abs{z}} \leq \abs{z}^{-(1-\varepsilon)}.
  \end{align*}
  Furthermore, $\alpha \int_J t^{\alpha-1} E_{\alpha,1}'(-\lambda_k t^\alpha)\,dt
  = \lambda_k^{-1} (1-E_{\alpha,1}(-\lambda_k T^\alpha))$, and
  $t^{\alpha-1}E_{\alpha,1}'(-\lambda_kt^\alpha)\geq 0$ due to~\eqref{eq:mon}.
  Hence, we see
  \begin{align}\label{eq:ml:est1}
    \int_J \abs{t^{\alpha-1}E_{\alpha,1}'(-\lambda_kt^\alpha)}\,dt \lesssim \lambda_k^{-1}.
  \end{align}
  We conclude that for $2\alpha(1-\varepsilon)<1$, it holds
  \begin{align}\label{eq:ml:est2}
    \norm{E_{\alpha,1}(-\lambda_k (\cdot)^\alpha)}{L_2(J)}^2 \leq C_\varepsilon \lambda_k^{-2(1-\varepsilon)}.
  \end{align}
  Furthermore, due to Lemma~\ref{lem:equivalence}, the identity~\eqref{eq:d:eig}
  and the previous estimate, we also have
  \begin{align}\label{eq:ml:est3}
    \snorm{E_\alpha(-\lambda_k (\cdot)^\alpha)}{H^\alpha(J)}^2
    \sim \lambda_k^2 \norm{E_{\alpha,1}(-\lambda_k (\cdot)^\alpha)}{L_2(J)}^2
    \leq C_\varepsilon \lambda_k^{2\varepsilon}.
  \end{align}
  By Young's inequality and~\eqref{eq:ml:est1},
  \begin{align}\label{eq:ml:est4}
    \norm{\psi_k}{L_2(J)}^2 &\lesssim
    \left( \int_J \dual{f(t)}{w_k}^2\,dt \right) \cdot
    \left( 
    \int_J \abs{t^{\alpha-1} E_\alpha'(-\lambda_kt^\alpha)}\,dt
    \right)^2
    \lesssim \lambda_k^{-2} \int_J \dual{f(t)}{w_k}^2\,dt.
  \end{align}
  According to~\cite[pp.~140]{Podlubny},
  it holds $\partial_t^\alpha \psi_k(t) = -\dual{f(t)}{w_k} - \lambda_k \psi_k(t)$.
  Hence, using Lemma~\ref{lem:equivalence} and~\eqref{eq:ml:est4}, we also see
  \begin{align}\label{eq:ml:est5}
    \snorm{\psi}{H^\alpha(J)}^2 &\lesssim
    \int_J \dual{f(t)}{w_k}^2\,dt.
  \end{align}
  Now choose $\varepsilon := (2-1/\alpha+\delta)/2$ and observe that $2\alpha(1-\varepsilon)=1-\alpha\delta <1$
  and $-1+2\varepsilon = 1-1/\alpha+\delta$. Using~\eqref{eq:ml:est2} and~\eqref{eq:ml:est4}, we estimate
  \begin{align*}
    \norm{u_m}{L_2(J;\wilde H^1(\Omega))}^2 = \sum_{k=0}^m \lambda_k \norm{d_m^k}{L_2(J)}^2
    &\lesssim \sum_{k=0}^m \lambda_k^{-1+2\varepsilon} \dual{g}{w_k}^2
    + \sum_{k=0}^m \lambda_k^{-1}\int_J \dual{f(t)}{w_k}^2\,dt\\
    &\lesssim \norm{g}{H^{1-1/\alpha+\delta}(\Omega)}^2 + \norm{f}{L_2(J;H^{-1}(\Omega))}^2.
  \end{align*}
  Using~\eqref{eq:ml:est3} and~\eqref{eq:ml:est5}, we can analogously estimate
  \begin{align*}
    \norm{u_m}{H^\alpha(J;H^{-1}(\Omega))}^2 
    \lesssim \norm{g}{H^{1-1/\alpha+\delta}(\Omega)}^2 + \norm{f}{L_2(J;H^{-1}(\Omega))}^2.
  \end{align*}
  Therefore, $(u_m)_{m\in\N}$ is a bounded sequence in $L_2(J;\wilde H^1(\Omega))$ and in
  $H^\alpha(J;H^{-1}(\Omega))$, and we conclude that there is
  a subsequence $(u_{m_k})_{k\in\N}$ which converges weakly to some $u\in L_2(J;\wilde H^1(\Omega))$
  and to some $\widetilde u\in H^\alpha(J;H^{-1}(\Omega))$.
  It follows that $u$ also converges weakly in $L_2(J;H^{-1}(\Omega))$ to $u$ as well as to $\widetilde u$,
  which yields $u=\widetilde u$. Taking into account the construction of the $u_m$ and invoking the weak limit,
  we obtain for all $v\in L_2(J;\wilde H^1(\Omega))$
  \begin{align*}
    \int_J \dual{\partial_t^\alpha u}{v} + \dual{\nabla u}{\nabla v}\,dt = \int_J\dual{f}{v}\,dt.
  \end{align*}
  Note that due to Corollary~\ref{cor:embedding}, $u_{m_k}$ also converges
  weakly to $u$ in $C([0,T];H^{1-1/\alpha+\delta}(\Omega))$,
  hence $g=u_{m_k}(0)\rightarrow u(0)$. This yields $u(0) = g$, and we conclude that $u$ is a weak solution.
  As for uniqueness, if $u$ is a weak solution with vanishing data, then the
  functions $u_k(t) := \vdual{u(t)}{w_k}$ solve the equations~\eqref{thm:eq:fde}
  with vanishing right-hand side, and hence $u_k(t)=0$.
\end{proof}
\textbf{Acknowledgement:} The author would like to thank Vincent~J.~Ervin for his valuable comments.
\bibliographystyle{abbrv}
\bibliography{literature}

\begin{thebibliography}{10}

\bibitem{AllenCV_16}
M.~Allen, L.~Caffarelli, and A.~Vasseur.
\newblock A parabolic problem with a fractional time derivative.
\newblock {\em Arch. Ration. Mech. Anal.}, 221(2):603--630, 2016.

\bibitem{Barrett_54}
J.~H. Barrett.
\newblock Differential equations of non-integer order.
\newblock {\em Canadian J. Math.}, 6:529--541, 1954.

\bibitem{BrennerScott_08}
S.~C. Brenner and L.~R. Scott.
\newblock {\em The mathematical theory of finite element methods}, volume~15 of
  {\em Texts in Applied Mathematics}.
\newblock Springer, New York, third edition, 2008.

\bibitem{CaffarelliS_07}
L.~Caffarelli and L.~Silvestre.
\newblock An extension problem related to the fractional {L}aplacian.
\newblock {\em Comm. Partial Differential Equations}, 32(7-9):1245--1260, 2007.

\bibitem{Diethelm_FracDiff10}
K.~Diethelm.
\newblock {\em The analysis of fractional differential equations}, volume 2004
  of {\em Lecture Notes in Mathematics}.
\newblock Springer-Verlag, Berlin, 2010.
\newblock An application-oriented exposition using differential operators of
  Caputo type.

\bibitem{ErdelyiMOT_81}
A.~Erd\'elyi, W.~Magnus, F.~Oberhettinger, and F.~G.~a. Tricomi.
\newblock {\em Higher transcendental functions. {V}ol. {III}}.
\newblock Robert E. Krieger Publishing Co., Inc., Melbourne, Fla., 1981.
\newblock Based on notes left by Harry Bateman, Reprint of the 1955 original.

\bibitem{EFHK_17}
V.~Ervin, T.~F{\"u}hrer, N.~Heuer, and M.~Karkulik.
\newblock {DPG} method with optimal test functions for a fractional advection
  diffusion equation.
\newblock {\em J. Sci. Comput.}, 2017.
\newblock Accepted for publication.

\bibitem{ErvinHR_07}
V.~J. Ervin, N.~Heuer, and J.~P. Roop.
\newblock Numerical approximation of a time dependent, nonlinear,
  space-fractional diffusion equation.
\newblock {\em SIAM J. Numer. Anal.}, 45(2):572--591, 2007.

\bibitem{ErvinR_06}
V.~J. Ervin and J.~P. Roop.
\newblock Variational formulation for the stationary fractional advection
  dispersion equation.
\newblock {\em Numer. Methods Partial Differential Equations}, 22(3):558--576,
  2006.

\bibitem{Heuer}
N.~Heuer.
\newblock Additive {S}chwarz method for the {$p$}-version of the boundary
  element method for the single layer potential operator on a plane screen.
\newblock {\em Numer. Math.}, 88(3):485--511, 2001.

\bibitem{HNVW_16}
T.~Hyt\"onen, J.~van Neerven, M.~Veraar, and L.~Weis.
\newblock {\em Analysis in Banach Spaces}.
\newblock Springer International Publishing, 2016.
\newblock Volume I: Martingales and Littlewood-Paley Theory.

\bibitem{JLPR_15}
B.~Jin, R.~Lazarov, J.~Pasciak, and W.~Rundell.
\newblock Variational formulation of problems involving fractional order
  differential operators.
\newblock {\em Math. Comp.}, 84(296):2665--2700, 2015.

\bibitem{JinLZ_13}
B.~Jin, R.~Lazarov, and Z.~Zhou.
\newblock Error estimates for a semidiscrete finite element method for
  fractional order parabolic equations.
\newblock {\em SIAM J. Numer. Anal.}, 51(1):445--466, 2013.

\bibitem{JinLZ_16}
B.~Jin, R.~Lazarov, and Z.~Zhou.
\newblock Two fully discrete schemes for fractional diffusion and
  diffusion-wave equations with nonsmooth data.
\newblock {\em SIAM J. Sci. Comput.}, 38(1):A146--A170, 2016.

\bibitem{KilbasST_06}
A.~A. Kilbas, H.~M. Srivastava, and J.~J. Trujillo.
\newblock {\em Theory and applications of fractional differential equations},
  volume 204 of {\em North-Holland Mathematics Studies}.
\newblock Elsevier Science B.V., Amsterdam, 2006.

\bibitem{Marcinkiewicz_39}
J.~Marcinkiewicz and A.~Zygmund.
\newblock Quelques in{\'e}galit{\'e}s pour les op{\'e}rations lin{\'e}aires.
\newblock {\em Fundamenta Mathematicae}, 32(1):115--121, 1939.

\bibitem{McLeanST_06}
W.~McLean, I.~H. Sloan, and V.~Thom\'ee.
\newblock Time discretization via {L}aplace transformation of an
  integro-differential equation of parabolic type.
\newblock {\em Numer. Math.}, 102(3):497--522, 2006.

\bibitem{MetzlerK_RandomWalkGuide_00}
R.~Metzler and J.~Klafter.
\newblock The random walk's guide to anomalous diffusion: a fractional dynamics
  approach.
\newblock {\em Phys. Rep.}, 339(1):77, 2000.

\bibitem{Mustapha_15}
K.~Mustapha.
\newblock Time-stepping discontinuous {G}alerkin methods for fractional
  diffusion problems.
\newblock {\em Numer. Math.}, 130(3):497--516, 2015.

\bibitem{MustaphaNC_16}
K.~Mustapha, M.~Nour, and B.~Cockburn.
\newblock Convergence and superconvergence analyses of {HDG} methods for time
  fractional diffusion problems.
\newblock {\em Adv. Comput. Math.}, 42(2):377--393, 2016.

\bibitem{NochettoO_16}
R.~H. Nochetto, E.~Ot\'arola, and A.~J. Salgado.
\newblock A {PDE} approach to space-time fractional parabolic problems.
\newblock {\em SIAM J. Numer. Anal.}, 54(2):848--873, 2016.

\bibitem{Podlubny}
I.~Podlubny.
\newblock {\em Fractional differential equations}, volume 198 of {\em
  Mathematics in Science and Engineering}.
\newblock Academic Press, Inc., San Diego, CA, 1999.
\newblock An introduction to fractional derivatives, fractional differential
  equations, to methods of their solution and some of their applications.

\bibitem{Pruess_12}
J.~Pr\"uss.
\newblock {\em Evolutionary integral equations and applications}.
\newblock Modern Birkh\"auser Classics. Birkh\"auser/Springer Basel AG, Basel,
  1993.
\newblock [2012] reprint of the 1993 edition.

\bibitem{SakamotoY_11}
K.~Sakamoto and M.~Yamamoto.
\newblock Initial value/boundary value problems for fractional diffusion-wave
  equations and applications to some inverse problems.
\newblock {\em J. Math. Anal. Appl.}, 382(1):426--447, 2011.

\bibitem{Schneider_96}
W.~R. Schneider.
\newblock Completely monotone generalized {M}ittag-{L}effler functions.
\newblock {\em Exposition. Math.}, 14(1):3--16, 1996.

\bibitem{Triebel}
H.~Triebel.
\newblock {\em Interpolation theory, function spaces, differential operators}.
\newblock Johann Ambrosius Barth, Heidelberg, second edition, 1995.

\bibitem{Zacher_09}
R.~Zacher.
\newblock Weak solutions of abstract evolutionary integro-differential
  equations in {H}ilbert spaces.
\newblock {\em Funkcial. Ekvac.}, 52(1):1--18, 2009.

\end{thebibliography}
\end{document}